\documentclass[11pt]{article}
\usepackage{amsfonts}
\usepackage{amsmath}
\usepackage{amsthm}
\usepackage{amssymb}
\usepackage{mathrsfs}
\usepackage{color}
\usepackage[top=2.5cm,bottom=3cm,right=2.5cm,left=2.5cm]{geometry}

\newcommand{\E}{\mathbb{E}}
\newcommand{\N}{\mathbb{N}}

\newcommand{\R}{\mathbb{R}}
\newcommand{\Z}{\mathbb{Z}}

\theoremstyle{plain}
\newtheorem{theorem}{Theorem}[section]

\newtheorem{lemma}[theorem]{Lemma}

\theoremstyle{definition}
\newtheorem{definition}[theorem]{Definition}

\usepackage{color}

\begin{document}

\title {On quantitative noise stability and influences for discrete and continuous models}
\author{\sc{Rapha{\"e}l Bouyrie } \\ 
\\ \textit{University of Marne--La--Vall\'ee, France}}
\date{}
\maketitle

\begin{abstract}
\textit{Keller and Kindler recently established a quantitative version of the famous Benjamini~--Kalai--Schramm Theorem on noise sensitivity of Boolean functions. The result was extended to the continuous Gaussian setting by Keller, Mossel and Sen by means of a Central Limit Theorem argument. In this work, we present an unified approach of these results, both in discrete and continuous settings. The proof relies on semigroup decompositions together with a suitable cut-off argument allowing for the efficient use of the classical hypercontractivity tool behind these results. It extends to further models of interest such as families of log-concave measures and Cayley and Schreier graphs. In particular we obtain a quantitative version of the B-K-S Theorem for the slices of the Boolean cube.} \\

\end{abstract}

\section{Introduction}

The notion of influences of variables on Boolean functions has been extensively studied
over the last twenty years, with applications in various areas such as combinatorics, statistical physics
and theoretical computer science, in particular cryptography and computational lower bounds
(see e.g. the survey \cite{K-S}). 
Similarly, the noise sensitivity of a Boolean function is a measure of how its values are likely to change under a slightly perturbed input.
Noise sensitivity became an important concept which is useful in many fields such as for
instance percolation, from which 
it was originally defined in \cite{B-K-S}. For an overview of the topic, see e.g. the book \cite{G-S}, in which the noise sensitivity concept 
and the results of \cite{B-K-S} are presented in the context of percolation theory.

Noise sensitivity and influences are closely related.
In this work, we will be concerned with recent connections between
influences and asymptotic noise sensitivity. 
To start with, let us recall these two important concepts on the discrete cube $\{-1,1\}^n$.
Rather than noise sensitivity, we describe the related notion of
noise stability. Throughout this paper, denote by $\nu$ the product measure
$(p \delta_{-1} + (1-p) \delta_1)^{\otimes n}$ on $\{-1,1\}^n$ where $p \in (0,1)$.

Let $A \subset \{-1,1\}^n$. For $x = (x_1, \ldots, x_n) \in \{-1,1\}^n$
and $i = 1, \ldots, n$, let $\tau_i x \in \{-1,1\}^n$ be the vector obtained 
from $x$ by changing $x_i$ in $-x_i$ and leaving the other coordinates unchanged.
The \textit {influence} of the $i$-th coordinate on a function $f : \{-1,1\}^n \to \R$ is given by 
$$
I_i (f) = \| f(x) - f(\tau_i x) \|_{L^1 (\nu)}, 
$$ 
Similarly, for sets $A \subset \{-1,1\}^n$, the influence of the $i$-th coordinate is defined using characteristic functions by $I_i(A) = I_i({\bf 1}_A)$. Notice then that in the uniform case ($p = 1/2$), it holds
$$
I_i (A) = 2 \nu \big \{x \in \{-1, 1\}^n; \; x \in A, \, \tau_i x \notin A  \big \}.
$$

Turning to noise stability, let $\eta \in (0,1)$ and let
$X = (X_1, \ldots, X_n)$ be distributed according to $\nu$ on
$\{-1,1\}^n$. Let $X^{\eta} = (X_1^{\eta}, \ldots, X_n^{\eta})$ be a
$(1-\eta)$-correlated copy of $X$, that is $ X_j^{\eta} =  X_j$ with probability $1-\eta$ and
$ X_j^{\eta} =  X_j'$ with probability $\eta$ where $X'$ is an independent copy of $X$. 
For a function $f : \{-1,1\}^n \to \R $, following the notations of \cite{K-M-S2}, define its \textit{noise stability}  as
$$
\mathcal{S}_{\eta}^{c} (f) 
= \E_{\nu} [f(X)f(X^{\eta})] - \E_{\nu} [f(X)]^2.
$$

Similarly, for a subset $ B \subset \{-1,1\}^n$, define its noise stability by $\mathcal{S}_{\eta}^{c} (B)  := \mathcal{S}_{\eta}^{c} ({\bf 1}_B)$.  Thus, the noise stability of a function $\mathcal{S}_{\eta}^{c}(f)$ is a measure of the sensitivity of $f$ to a small noise $\eta$ in its input. A sequence of functions/sets $(f_{n_{\ell}})_{n_{\ell}}, (B_{n_{\ell}})_{n_{\ell}}$ over  $\{-1,1\}^{n_{\ell}}$, $ \ell \in \N$, where $n_\ell \nearrow \infty$, is said to be \textit{(asymptotically) noise sensitive} if its noise stability tends to $0$ as $\ell $ tends
to $\infty$, that is 
$$
\lim_{\ell \to \infty} \mathcal{S}_{\eta}^{c} (f_{n_{\ell}}) = 0
$$ 
or
$$
\lim_{\ell \to \infty} \mathcal{S}_{\eta}^{c} ( B_{n_{\ell}}) = 0
$$ 
for each $\eta \in (0,1)$ (keeping in mind $\eta$ small but fixed).

A first connection between noise sensitivity and influences have been established by
Benjamini, Kalai and Schramm in the paper \cite{B-K-S}.
They gave a criterion for a sequence a sets to be noise sensitive in terms of the sum of squares of the influences. More precisely, one of the main results of \cite{B-K-S} is the following theorem.
\begin{theorem} \label {bks}
Let $B_{\ell} \subset \{-1,1\}^{n_{\ell}}$, $ \ell \in \N$. If
\begin{equation}
\lim_{\ell \rightarrow \infty} \sum_{i=1}^{n_{\ell}} I_i(B_{\ell})^2 = 0, \label{eq.1}
\end{equation}
then  ${(B_{\ell})}_{\ell \in \N}$ is asymptotically noise sensitive.
\end{theorem}

In fact, BKS only proved this theorem for the uniform measure (i.e., in the case $p = 1/2$). The general result, including the case of biased measure (assuming the bias $p$ does not depend on $n$), was proved recently by Keller and Kindler \cite{K-K}, who moreover established a quantitative version of this result. The proof of this quantitative bound
uses the Fourier--Walsh expansion of functions of the discrete cube together with hypercontractive bounds to control the tails of chaos in the expansion (hypercontractivity is actually also at the root of the asymptotic Theorem~\ref {bks}).
The result is stated below as \eqref {eq.noise1},
but before we present the analogous
continuous version which has a similar form.

The quantitative version of \cite{K-K} has been extended to 
the continuous Gaussian setting by Keller, Mossel and Sen \cite{K-M-S2}.
To state the result, we need to introduce the corresponding definitions. 
Let $d\mu (x) = e^{-|x|^2/2} \frac {dx}{(2\pi)^{n/2}}$
be the canonical Gaussian measure on $\R^n$. For $W, W'$ independent
with distribution $\mu$, and $\eta >0$,
set 
$$
W^{\eta} = \sqrt{1-\eta^2} \, W + \eta W',
$$ 
so that $(W,W^{\eta})$ is a $\sqrt{1-\eta^2}$-correlated Gaussian vector. 
For $f : \, \R^n \to \R$, providing that $\| f(W) \|_{L^2(\mu)} < \infty$,
define the Gaussian noise stability of $f$ as 
$$
\mathcal{S}_{\eta}^{\mathcal{G}} (f) = \E_{\mu} [f(W)f(W^{\eta})] - \E_{\mu} [f(W)]^2.
$$  
Similarly, for a (Borel measurable) subset $A \subset \R^n$, set
$$
\mathcal{S}_{\eta}^{\mathcal{G}} (A) = \mathcal{S}_{\eta}^{\mathcal{G}} ({\bf 1}_A).
$$

In order to discuss the analogue of Theorem~\ref {bks} in the Gaussian case, 
it is necessary to define the
influence of the $i$-th coordinate on a subset $A \subset \R^n$ in this context. To this aim, in a continuous
setting, the authors of \cite{K-M-S2} introduce the notion of geometric influence  defined by
$$
I_i^{\mathcal{G}} (A) = \E_{x} [\mu^+ (A_i^x)].
$$
In the latter expression, $A_i^x \subset \R$ is the restriction of $A$ along the fibre of
$x = (x_1, \ldots, x_n) \in \R^n$, that is 
$$
A_i^x = \{y \in \R, \, (x_1, \ldots, x_{i-1},y, x_{i+1}, \ldots, x_n) \in A \}
$$
and $\mu^+$ denotes the lower Minkowski content, that is for 
any Borel measurable set $D \subset\R$,
$$
\mu^+(D) =  \liminf\limits_{r \to 0} \frac{\mu(D+[-r,r]) - \mu(D)}{r}
$$
(where $\mu$ here is the standard Gaussian distribution on $\R$).

From a more intuitive point of view, for each $i \in \{1, \ldots, n \}$, $I_i^{\mathcal{G}} (A)$ is obtained as a limit of $\|\partial_i f_{\varepsilon_k}  \|_{L^1(\mu)}$ for a sequence of smooth functions $(f_{\varepsilon_k})_{\varepsilon_k \geq 0}$ such that $\lim_{\varepsilon_k \to 0} f_{\varepsilon_k} = {\bf 1}_A$.

We refer the reader to \cite{K-M-S1}, \cite{K-M-S2} for further developments on geometric influences and their applications.

In both the cube (with bias $p$) and Gaussian settings, the quantitative noise sensitivity bounds 
of \cite{K-K} and \cite{K-M-S2} may be then expressed in the following form: for any sets $A \subset \{-1,1\}^n,$  and any $\eta \in (0,1)$,
\begin{equation} \label{eq.noise1}
 \mathcal{S}_{\eta}^{c}(A) 
    \leq C_1 \bigg( \sum_{i=1}^n  I_i (A)^2 \bigg)^{C_2(p) \eta} 
\end{equation}
for the discrete cube where $C_1$ is numerical, and for any sets $A \subset \R^n$ and any $\eta \in (0,1)$, 
\begin{equation} \label{eq.noise2}
 \mathcal{S}_{\eta}^{\mathcal{G}} (A) 
    \leq C_1 \bigg( \sum_{i=1}^n (I_i^{\mathcal{G}} (A))^2 \bigg)^{C_2 \eta^2} 
\end{equation}
for the Gaussian space, where $C_1, \, C_2$ are some positive numerical constants.  
These results actually are proved in theirs functional forms, that is 
for any bounded (by $1$) function $f$ defined respectively in the discrete cube and in the Gaussian space ($C^1$ smooth in the latter case), it holds
$$
\mathcal{S}_{\eta}^{c} (f) 
    \leq C_1 \bigg( \sum_{i=1}^n  I_i (f)^2 \bigg)^{C_2 \eta}  \quad \mathrm{and} \quad   \mathcal{S}_{\eta}^{\mathcal{G}} (f) 
    \leq C_1 \bigg( \sum_{i=1}^n \| \partial_i f \| _1^2 \bigg)^{C_2 \eta^2}.
$$
   
The results in \cite{K-K} and \cite{K-M-S2} actually extend for functions in $L^2(\mu)$.
Indeed, the proof of the Gaussian result in \cite{K-M-S2} relies 
on the result of the discrete cube in \cite{K-K} together
with an appropriate Central Limit Theorem argument. 
A close inspection of the arguments from \cite{K-K} actually reveals that the Fourier--Walsh
decomposition approach may be adapted to a Fourier--Hermite decomposition in the Gaussian case
so to yield the same conclusion. Therefore, the boundedness of the functions
can be weakened to obtain the same result for functions in $L^2(\mu)$.

The Fourier decomposition approach is however
somewhat limited to the examples of the discrete cube and the Gaussian space. Indeed, a key property is
the fact that the elements of the orthogonal basis of the underlying space are eigenvectors of the 
corresponding semigroup. Spaces satisfying such a property have been characterized in \cite{Ma}. In this paper, 
we develop a new, simpler, proof of these quantitative relationships \eqref {eq.noise1} and \eqref {eq.noise2} 
which moreover applies to more general examples. Once again, the main ingredient will be hypercontractivity
as in \cite{B-K-S}, \cite{K-K}, \cite{K-M-S2}. The starting point of the 
proof follows closely the work of \cite{CE-L} which generalizes  
Talagrand's inequality of \cite{T1} for more general models. The framework of \cite{CE-L} applies in a rather 
general context where hypercontractivity holds together with specific commutation
properties. It covers the product of (strictly) log-concave probability measures and
also discrete examples, including the biased discrete cube or, more general Schreier or Cayley graphs.

In this framework, we will establish quantitative relationships between noise stability and 
influences strengthening therefore the results 
of \cite{B-K-S}, \cite{K-K} and \cite{K-M-S2}, under furthermore weaker assumptions.  
The proposed simpler and more efficient proof relies on semigroup decompositions and 
cut-off arguments. 

The general setting contains two main illustrations,
probability measures on finite state spaces that are invariant for some Markov kernel 
and continuous product probability 
measures on $\R^n$ where each measure is of the form $d\mu(x) = e^{-V(x)}dx$ for $V$ a smooth potential.  The following is a sample illustration of the results of this work
(in the continuous setting). More complete statements will be presented in the process of
the paper.

\begin{theorem} \label {blg1}
Let $(\R^n, \mu^{\otimes n})$ with $\mu$ a probability measure on $\R$ of the form $ d\mu (x) = e^{-v(x)} dx$
where $ v''\geq c$ uniformly for some $c > 0$. Denote by $\mathcal{S}_{\eta}$ the noise stability with respect to $\eta$ in this context.  
Then, for any $\eta \in [0,1)$, and any Borel measurable set $A \in \R^n$, 
there exists positive constants $C, c_1$ depending only on $c$ such that
$$
\mathcal{S}_{\eta} (A)  \leq C 
   \bigg( \sum_{i=1}^n (I_i^{\mathcal{G}} (A)) ^2 \bigg)^{c_1 \eta^2} \mu^{\otimes n}(A)^{2 - 2c_1 \eta^2}.
$$
\end{theorem}

The generalized definition of the noise stability in this context will be given in the next section. 
In particular, for the standard Gaussian measure, $\mathcal{S}_{\eta} = \mathcal{S}_{\eta}^{\mathcal{G}}$ and so we recover the quantitative estimate \eqref{eq.noise2} 
established by Keller, Mossel and Sen.

With a common scheme of proof, an analogous statement is established for discrete models
covering in particular the discrete cube
endowed with any biased measure ${(p \delta_{-1} + (1-p) \delta_1)^{\otimes n}}$. The precise setting and the according notion of influence will be presented in subsection $2.3$.  

\begin{theorem} \label {bld1}
Let $(\Omega, \mu)  = (\Omega_1 \times \ldots \times \Omega_n, \mu_1 \otimes \ldots \otimes \mu_n)$ be a Cartesian product of finite probability spaces. Denote by $\mathcal{S}_{\eta}$ the noise stability with respect to $\eta$ in this context.  
Then, for any $\eta \in [0,1)$, and any set $A \subset \Omega$, 
there exists some numerical constants $C, c_1 >0$  such that
$$
\mathcal{S}_{\eta} (A)  \leq C 
   \bigg( \sum_{i=1}^n (I_i (A)) ^2 \bigg)^{c_1 \frac{\rho}{\lambda} \eta} \mu (A)^{2 - 2 \frac{\rho}{\lambda} c_1 \eta},
$$
where $\lambda$ and $\rho$ designed respectively the spectral gap and the Sobolev logarithmic constants of the product graph $\Omega$.
\end{theorem}
The definitions of $\lambda$ and $\rho$ will be given in subsections $2.4$ and $2.5$. In particular if $\Omega = \{-1,1\}^n$, $\mu = \nu$, $\mathcal{S}_{\eta}  = \mathcal{S}_{\eta}^c$ and we will see that $\rho = 2\frac{\log p - \log (1-p)}{2p -1}$, $\lambda =1$ so that the above theorem extends the quantitative estimate \eqref{eq.noise1} of Keller and Kindler.

\vskip 2mm

The note is organized as follows. In Section~2, we first describe a convenient abstract framework and recall some basic facts about Markov semigroups that will be used in the proofs of our results. Then, in subsection $2.1, 2.2$ and $2.3$, after developing the examples of the Gaussian space and the discrete cube, we present the general setting in which our results fall in. Then, we present our generalized  definition of noise stability in subsection $2.4$  and finally in subsection $2.5$ we describe further tools required in the proofs of our results, in particular the hypercontractive property.
In Section~3, we establish the announced result on ($\R^n, \, d\mu(x) =  \bigotimes_{i=1}^m e^{-V_i(x)}dx$) 
when the potentials $V_i$ are convex. This more general result encircles Theorem \ref {blg1}. Section~4 is devoted to the discrete case. Firstly, we focus on product spaces, proving Theorem \ref {bld1} in a slightly more general form, and then we turn to the case of Cayley or Schreier graphs. In particular, we prove that the analogous result holds for the slices of the Boolean cube. We then briefly conclude in Section~5 with a similar inequality on the Euclidean spheres.

\section{A general framework}

This section presents the framework and the main tools that will be required in the proofs. The setting
emphasized here is pretty general, but for the
sake of clarity, we discuss in subsections the two main cases that are at the starting point of 
this investigation, i.e. the Gaussian space and the discrete cube endowed with uniform measure.

Let $(\Omega, \mathcal{A}, \mu)$ be a probability space.
For a function $f : \Omega \to \R$ in $L^2(\mu)$, 
denote its variance with respect to $\mu$ by
$$ 
\mathrm{Var}_{\mu} (f) = \int_{\Omega} f^2 d\mu - \bigg( \int_{\Omega} f d\mu \bigg)^2.
$$ 
In the same way, if $f \geq 0$, provided it is well defined, we denote its entropy with respect to $\mu$ by
$$ 
\mathrm{Ent}_{\mu} (f) = \int_{\Omega} f \log f d\mu 
     - \int_{\Omega} f d\mu \log \bigg(  \int_{\Omega} f d\mu \bigg).
$$ 

\par

The main argument of the proof will be based on interpolation along a Markov semigroup with 
invariant measure $\mu$. We refer to the general
references \cite{Ba}, \cite{Aal}, \cite{B-G-L} for complete accounts on
Markov semigroups. For the reader's convenience, we briefly recall here a few basic aspects,
illustrated next on the two basic model examples.

A family $(P_t)_{t \geq 0}$ of operators
acting on a domain $\mathcal{D}$ of functions
on $\Omega$ is said to be a semigroup if $P_0 = \mathrm{Id}$ and, for all  $ s,t \geq 0$,
$P_{t+s} = P_t \circ P_s$. The semigroup $(P_t)_{t \geq 0}$ 
is said to be Markov if for all $ t \geq 0$, $ P_t {\bf 1} = {\bf 1}$.
The infinitesimal generator $L$ of $(P_t)_{t \geq 0}$  is defined by 
$$ 
\forall f \in \mathcal{D}_2(L), \, Lf := \lim_{t \to 0} \frac{P_t f - f}{t}
$$
where the Dirichlet domain $\mathcal{D}_2(L) \subset \mathcal{D}$
is the set of all functions $f$ in $L^2(\mu)$ for which the above limit exists. 
Conversely, $L$ and $\mathcal{D}_2 (L)$ completely determine $(P_t)_{t \geq 0}$.
By definition and the semigroup property 
$\frac{\partial}{\partial t} P_t f = L P_t f$ and $P_0 f = f$, justifying
the intuitive notation $P_t = e^{tL}$. 

Given such a Markov semigroup $(P_t)_{t \geq 0}$, the 
measure $\mu$ is said to be reversible with respect to $(P_t)_{t \geq 0}$ if
$$
\forall f, g \in L^2(\mu), \, \int_{\Omega} f Lg \, d\mu = \int_{\Omega} g Lf \, d\mu,
$$ 
and invariant with respect to $(P_t)_{t \geq 0}$ if 
$$
\forall f \in L^1(\mu), \, \int_{\Omega} P_t f  \, d\mu = \int_{\Omega} f \, d\mu.
$$
The Dirichlet form associated to $(L, \mu)$ is the bilinear symmetric operator
$$ 
\mathcal{E}(f,g) = \int_{\Omega} f(-Lg) d\mu $$
on suitable real-valued functions $f,g$ in the Dirichlet domain. 

Finally, it will be assumed moreover that
$(P_t)_{t \geq 0}$ is ergodic with respect to $\mu$, which means that 
$\mu$ a-e, $P_t f~\to~\int_{\Omega} f d\mu$ as $t \to \infty$. We notice then 
-- as a basic starting point of the future investigation -- that the variance of function $f$
with respect to $\mu$ can be represented via the semigroup as
$$
\mathrm{Var}_{\mu} (f) = \int_{\Omega} f^2 d\mu - \bigg( \int_{\Omega} f  d\mu \bigg) ^2 
= \lim_{t \to \infty} \bigg (\int_{\Omega} (P_0 f)^2 d\mu - \int_{\Omega} (P_t f)^2 d\mu \bigg) .
$$

The following paragraphs aim at illustrating this general set-up by examples of interest.
The first one discusses the Gaussian model, and its extension to log-concave measures.
The next ones deal with the discrete cube and more general discrete models attached to
Markov chains.

\vskip 2mm 

\subsection{The Gaussian space and continuous setting}

Let $d\mu (x) = e^{-|x|^2/2} \frac {dx}{(2\pi)^{n/2}}$ be the standard Gaussian measure
on $\Omega = \R^n$, and consider the
Ornstein-Uhlenbeck semigroup acting on suitable functions $f : \R^n \to \R$ as
$$
P_t f(x) = \int_{\R^n} f(e^{-t}x + \sqrt{1-e^{-2t}}y) d\mu (y), \quad t \geq 0, \, \, x \in \R^n.
$$
As it is classical, (cf.~e.g. \cite{Ba}, p. 4), the generator of the Ornstein-Uhlenbeck semigroup
$(P_t)_{t \geq 0}$ is given by  $ L = \Delta - x \cdot \nabla $.
The Ornstein-Uhlenbeck semigroup $(P_t)_{t \geq 0}$ is invariant and symmetric with respect to $\mu$,
and ergodic (as easily checked on the previous integral representation). 
That is, for any functions $f,g \in L^2(\mu)$ and any $t \geq 0$,
$\int_{\R^n} f P_t g d\mu = \int_{\R^n} P_t f  g d\mu $, and $P_t f \to f$ (in $L^2(\mu)$)
as $t \to \infty$.
The associated Dirichlet domain of contains $L^2(\mu) \cap C^{2}(\R^n)$, and it
follows from integration by parts that for $C^2$ functions $f,g$ on $\R^n$,
$$
\mathcal{E}(f,g) = \int_{\R^n} f(-Lg) d\mu = \int_{\R^n} \nabla f \cdot \nabla g \, d\mu.
$$
In particular, we have the following decomposition of the Dirichlet form along directions as 
$$
\mathcal{E}(f,f) =  \int_{\R^n} |\nabla f|^2 \, d\mu 
= \sum_{i=1}^n \int_{\R^n} (\partial_i f)^2 d\mu.
$$  

According to these properties, it is immediately checked that the (Gaussian)
noise stability $ \mathcal{S}_{\eta}^{\mathcal{G}} (f) $ of a function $ f : \R^n \to \R$
as described in the introduction may be reinterpreted in terms of the semigroup in the following way

\begin{lemma} \label{lou}
For $f : \R^n \to \R$ and $\eta >0$,
$$
\mathcal{S}_{\eta}^{\mathcal{G}} (f)  = \mathrm{Var}_{\mu}(P_{t/2} f),
$$
with $e^{-t} = \sqrt{1 - \eta^2}.$
\end{lemma}

\vskip 2mm 

The preceding Gaussian example may be amplified along the same lines to cover families
of log-concave measures on $\R^n$.
Indeed, let $d\mu(x) = e^{-V(x)}dx$ be a probability measure on the Borel sets of $\R^n$
where $V : \R^n \to \R$ is a smooth potential,
invariant and symmetric with respect to the second order diffusion operator
$L = \Delta - \nabla V \cdot \nabla$ with associated semigroup $P_t = e^{tL}$, $t \geq 0$.
As in the Gaussian case, integration by parts yields, for smooth functions $f,g$ on $\R^n$,
$$
\mathcal{E}(f,g) = \int_{\R^n} f(-Lg) d\mu = \int_{\R^n} \nabla f \cdot \nabla g \, d\mu,
$$
and therefore a similar decomposition of the Dirichlet form $\mathcal{E}(f,f)$.

We will be concerned more generally
with products of such measures, namely $\mu = \bigotimes_{i=1}^m \mu_i$ on 
$\R^n~=~\R^{n_1} \times \cdots \times \R^{n_m}$
where, for $i =1, \ldots, m$, $d\mu_i (x)$  is of the form $e^{-V_i(x)} dx$ with 
$V_i : \R^{n_i} \to \R $ some smooth potential.
The product generator $L$ of the $L_i$ is given by
$$
L = \sum_{i=1}^m \mathrm{Id}_{\R^{n_1}} \otimes \cdots \otimes  \mathrm{Id}_{\R^{n_{i-1}}} \otimes L_i 
\otimes \mathrm{Id}_{\R^{n_{i+1}}} \otimes \cdots \otimes  \mathrm{Id}_{\R^{n_m}}
$$
with associated (product) semigroup $(P_t)_{t \geq 0}$.
Setting $\nabla_i$ for the gradient in the direction $\R^{n_i}$, the Dirichlet form is decomposed into
$$
\mathcal{E}(f,f) =  \int_{\R^n} |\nabla f|^2 \, d\mu 
= \sum_{i=1}^m \int_{\R^n} |\nabla_i f|^2 d\mu.
$$  

In this context, we may then state the classical decomposition of the variance (and
accordingly of noise stability) along the semigroup which will be the
starting point of our investigation.

\begin{lemma} \label{dovg}
For every smooth $f : \R^n \to \R$ and every $t \geq 0$,
\begin{equation} \label{eq.91}
\mathrm{Var}_{\mu}(P_t f) = 2 \int_t^{\infty} \sum_{i=1}^m \int_{\R^n} |\nabla_i P_s f|^2 d\mu \, ds. 
\end{equation}
\end{lemma}

\begin{proof}
By ergodicity and the fundamental theorem of calculus,
\begin{eqnarray*} \label{eq.77}
\mathrm{Var}_{\mu}(P_t f) & = &  \int_{\R^n} (P_t f)^2 d\mu - \left( \int_{\R^n} f d\mu \right)^2 \\
& = & - \int_t^{\infty} \frac{d}{ds} \int_{\R^n} (P_s f)^2 d\mu \, ds \\
&=& - 2 \int_t^{\infty}  \int_{\R^n} P_s f L(P_s f) d\mu \, ds. \\
\end{eqnarray*}
Then, after integration by parts,
$$
\mathrm{Var}_{\mu}(P_t f)
=  2 \int_t^{\infty} \mathcal{E} (P_s f, P_s f) ds 
=  2 \int_t^{\infty} \sum_{i=1}^m \int_{\R^n} |\nabla_i P_s f|^2 d\mu \, ds
$$
from which the lemma follows.
\end{proof}

\subsection{The discrete cube}

On the discrete cube $C_n = \{-1,1\}^n$, let us consider first the case 
of the uniform measure $\nu$ ($ p = \frac {1}{2}$).
The biased case is an adaptation of this particular case, as
developed below.
Consider the Bonami--Beckner semigroup $(T_t)_{t \geq 0}$
acting on the discrete cube given in its explicit form by
$$
T_t f(x) = \int_{C_n} f(\omega) \prod_{i=1}^n (1+e^{-t}x_i \omega_i) d \nu (\omega),
  \quad t \geq 0, \, \, x \in C_n.
$$
It is classical and easily checked that $(T_t)_{t \geq 0}$ is a Markov semigroup with invariant
and reversible probability measure $\nu$. Its generator $L$ may be described as
$$
L = \frac{1}{2} \sum_{i=1}^n D_i
$$
where $D_i$ the $i$-th derivative of $f : C_n \to \R$ defined by 
$ D_i (f) (x) = f(\tau_i x) - f(x)$
with $x= (x_1, \ldots, x_n)$ and $\tau_i x$ as in the introduction. 

It is actually of some interest to emphasize the latter properties in the Fourier--Walsh
decomposition.
Denote $[n]$ for $\{1, \ldots, n \}$. $L^2(C_n, \nu)$ is an Euclidean space with respect to the standard scalar product $\langle \cdot , \cdot \rangle_{L^2(\nu)}$. It is well known (see e.g. [T1]) that there is an orthogonal basis, the Fourier--Walsh basis $(W_{S})_{S \subset [n]}$  defined by $W_S (x) = \prod_{i \in S} x_i$. 
Each function $f \in L^2 (\nu)$ can then be decomposed into $f = \sum_{S \subset [n]} \hat{f} (S) W_s $ with 
$\hat{f}(S) = \langle f, W_S \rangle_{L^2(\nu)}$. 
Using that 
$$
 \prod_{i=1}^n (1+e^{-t}x_i \omega_i)  = \sum_{S \subset [n]} e^{- t |S|} \prod_{i \in S} x_i \omega_i 
= \sum_{S \subset [n]} e^{- t |S|}  W_S (x) W_S (\omega)
$$
with $|S|$ to be the cardinal of the set $S$, it implies that
$$
T_t (f) = \sum_{S \subset [n]} e^{-t|S|} \hat{f} (S) W_s
$$
for every $t >0$. The semigroup property then easily follows.
Moreover,
$$
L W_{S}(x) = - \frac{1}{2} \sum_{i=1}^n (W_S( x) -W_S(\tau_i x)) = - |S| W_{S} (x)
$$ 
since $W_S(x) - W_S(\tau_i x) = 2  W_S(x) 1_{i \in S}(x)$ from which $L$ is seen
as the generator of $(T_t)_{t \geq 0}$.

The associated Dirichlet domain is $L^2(\nu)$ since there are no regularity assumptions.
Notice that the Dirichlet form $\mathcal{E}$ may again be decomposed along directions by
\begin{equation}
\mathcal{E} (f,f) = \frac{1}{4} \, \sum_{i=1}^n \int_{C^n} |D_i(f)|^2 d\nu. \label{eq.10}
\end{equation}
Besides, for a real-valued function $f$ and $i = 1, \ldots , n$, the influence
$I_i(f)$ of the $i$-th coordinate on the function $f$ as presented in the introduction is represented as
$$
I_i(f) = \| D_i f \|_{L^1(\nu)}.
$$
With this background, as in the Gaussian space, the noise stability may be expressed in terms of the
semigroup $(T_t)_{t \geq 0}$.

\begin{lemma} \label{lbb}
For $f : C_n \to \R$ and $\eta > 0$,
$$
S_{\eta}^c (f) = \mathrm{Var}_{\nu}(P_{t/2} f),
$$
with $e^{-t} = 1 - \eta.$
\end{lemma}

\vskip 2mm

The preceding construction may be developed
similarly on the biased case, i.e. when the cube is endowed with $ \nu_{p} = (p \delta_{-1} + q \delta_1)^{\otimes n},$ where $p+q =1.$
The Fourier--Walsh basis $(\omega_S)_{S \subset [n]}$ is given here by
$$
\left( \left(\frac{p}{q}\right)^{(n - |S|)/2}\prod_{i \in S} x_i  \right)_{S \subset [n]}.
$$
Set $L = 2 p q \sum_{i=1}^n D_i$. Arguing as in the case
$p = \frac {1}{2}$, if we set $T_t^p f = \sum_{S \subset [n]} e^{-t |S|} \hat{f}(S) \omega_S$,
then $L$ is the generator of $(T_t^p)_{t \geq 0}$. In the same way
$$
\mathcal{S}_{\eta} (f) = \mathrm{Var}_{\nu_p}(T_{t/2}^p f).
$$
In addition, the Dirichlet form $\mathcal{E}$ takes the same form
$$
\mathcal{E} (f,f) = pq \, \sum_{i=1}^n  \int_{\{-1,1\}^n} |D_i(f)|^2 d\nu_p. 
$$

For both the uniform and biased cube,
the decomposition of the variance along the semigroup is similar to the one
emphasized in the continuous setting in Lemma~\ref {dovg}.

\begin{lemma} \label{dovc}
For every $f : C_n \to \R$, every $t \geq 0$, and every $p \in (0,1)$,
\begin{equation} \label{eq.91}
\mathrm{Var}_{\nu_p}(P_t f) = 2 \int_t^{\infty} \sum_{i=1}^n \int_{C^n} |D_i T_s^p f|^2 d\nu \, ds. 
\end{equation}
\end{lemma}
\vskip 2mm

\subsection{General discrete case}

In this subsection, we discuss extensions of the discrete cube model to general discrete spaces, that we assume to be finite (as it will be the case for concrete illustrations).

Let $\Omega$ be a finite space with probability measure $\mu$
on which there is a Markov kernel $K$, invariant
and reversible with respect to $\mu$, i.e.
such that
$$
\forall (x, y) \in \Omega^2, \quad
 \sum_{x \in \Omega} K(x,y) \mu (x) = \mu (y) \quad \mathrm{and} \quad K(x,y) \mu(x) = K(y,x) \mu(y).
$$ 
Define $L$ by $L = K - Id$, generator of the semigroup $P_t = e^{tL}$, $t \geq 0$.
The associated Dirichlet form is given by
$$
\mathcal{E}(f,g) = \int_{\Omega} f(-Lg) d\mu = \frac{1}{2}\sum_{x,y \in \Omega}
 (f(x)-f(y))(g(x)-g(y)) K(x,y) \mu(y)
$$
for functions $f,g$ on $\Omega$.

\vskip 2mm

The discrete cube model enters this setting by a suitable choice of the kernel $K$.
Consider namely the operator given by $Lf = \int_{\Omega} f d\mu  - f$,
i.e $K f = \int_{\Omega} f d\mu,$ or $K = \mathrm{diag}(\mu(x))_{x \in \Omega}$.
In particular, a simple computation shows that
$$
\mathrm{Var}_{\mu} (f) = \int_{\Omega} f (-Lf) d\mu = \mathcal{E} (f,f).
$$

An interesting instance of the preceding, extending the case of the Boolean cube, is given by such product spaces with product measures 
$$
\Omega = \Omega_1 \times \cdots \times \Omega_n \quad \mathrm{with} 
 \quad \mu = \mu_1 \otimes \cdots \otimes \mu_n,
$$ 
when we take product of the above Markov operators. 
That is, set, for $i  =1, \ldots, n $, and $f : \Omega \to \R$
$ L_i f = \int_{\Omega_i} f d\mu_i - f $
and consider the generator on the product space given by
$$
Lf = \sum_{i=1}^n L_i f .
$$
In this case the Dirichlet form $\mathcal{E}$ may be decomposed as
\begin{equation}
\mathcal{E} (f,f) = \sum_{i=1}^n  \int_{\Omega_i} L_i(f)^2 d\mu_i. \label{eq.9}
\end{equation} 
This setting, which will be referred to in the following as discrete product structure,
covers the cube with
$\Omega_1  = \cdots = \Omega_n = \{-1, 1 \}$, each of which being equipped with the
measure $p \delta_{-1} + q \delta_{1}$. Among other relevant examples one can take for any
$i$, $\Omega_i = \Z/q\Z$ for any $q \geq 3$ endowed with uniform measure. In this context, the influence of the $i$-th coordinate is naturally defined as $\| L_i f \|_1$ (notice that both definitions over the Boolean cube agree up to a constant depending on the bias $p$). 

\vspace{2mm}

The preceding can be extended to more general Cayley or Schreier graphs
(see \cite{CE-L}, \cite{O-W}), covering therefore non-product examples.
Let $G$ be a (finite) group for which there
is a finite set of generators $S$, symmetric, i.e. $S^{-1} =S$, and stable by conjugacy.
Assume that $G$ is acting transitively on a finite set $\Omega$ and denote for $g \in G$, $x \in \Omega$, $x^g$
the action of $g$ on $x$. The associated Schreier graph is the graph of vertices $\Omega$ and 
edges $(x,y)$ if and only if there is a $s \in S$ such that $y = x^s$.  
A Cayley graph corresponds to the particular case $G =\Omega$, consisting therefore of  
the set of vertices $(g)_{g \in G}$ with edges $(g, gs)_{g \in G, s \in S}$. 
One basic example  is the symmetric group $\mathfrak{S}_n$ with generating
set transpositions $\mathcal{T}_n$. 

Given a Cayley of Schreier graph $G$, consider the
transition kernel $K$ given by
$$
K(g_1,g_2) = \frac{1}{|S|} 1_{S}(g_1g_2^{-1}), \quad g_1, g_2 \in G,
$$
corresponding to the random walk to nearest neighbour
and the uniform measure $\mu$ on $G$.
Again, it generates the family of semigroups $(P_t = e^{tL})_{t \geq 0}$, with $L = K - Id$. The associated Dirichlet forms $\mathcal{E}$ can be written as
$$
\mathcal{E}(f,f) = \frac{1}{2|S|} \sum_{g \in G} \sum_{s \in S} [f(gs) - f(g)]^2 \mu(g) = \frac{1}{2|S|} \sum_{s \in S} \|D_{s} f\|_{L^2(G)}^2, 
$$
where $D_s f : g \mapsto f(gs) - f(g)$ in the Cayley graph case and 
$$
\mathcal{E}(f,f) = \frac{1}{2|S|} \sum_{x \in \Omega} \sum_{s \in S} [f(x^s) - f(x)]^2 \mu(x) = \frac{1}{2|S|} \sum_{s \in S} \|D_{s} f\|_{L^2(\Omega)}^2,
$$
where $D_s f : x \mapsto f(x^s) - f(x)$  in the Schreier graph case.

Note finally that the influence of a generator element
$s \in S$ on a function $f$ is then naturally defined in this context by $\| D_s f \|_1$.

\vskip 2mm

\subsection{The generalized version of noise stability}

In this section, we extend the definition of noise stability to the different models presented in the preceding subsections.

Consider therefore the preceding setting of a probability space $(\Omega, {\cal A}, \mu)$
equipped with a semigroup $(P_t)_{t \geq 0}$ with generator $L$ and Dirichlet form ${\cal E}$,
invariant and symmetric with respect to $\mu$.

Say that the couple $(L, \mu)$
satisfies a spectral gap, or Poincar{\'e}, inequality whenever there exists $\lambda > 0$ such that
\begin{equation}
\lambda \mathrm{Var}_{\mu}(f) \leq \mathcal{E}(f,f) \label{eq.3}
\end{equation}
for every function $f$ on the Dirichlet domain. 
The spectral gap constant is the largest $\lambda$ such that $\eqref{eq.3}$ holds. It is equivalent to the fact that for every functions $f$ in $L^2(\mu)$, and every $t>0$, 
\begin{equation} \label{eq.89}
\mathrm{Var}_{\mu} (P_t f)  \leq e^{-2\lambda t}\mathrm{Var}_{\mu} (f).
\end{equation}
This follows immediately from Gronwall's lemma since $\frac{d}{dt} \mathrm{Var}_{\mu} (P_t f)  = \frac{d}{dt} \|P_t f \|_2^2  = - 2 \mathcal{E}(P_t f, P_t f)$. For further purposes, it is then not hard to check (see \cite{CE-L}), that the spectral gap inequality with constant $\lambda$ is equivalent 
to the fact that for every centered function $f$ (i.e. $\int_{\Omega} f d\mu =0$),
\begin{equation}
\forall t >0, \, \mathrm{Var}_{\mu}(f) \leq \frac{1}{1-e^{-\lambda t}} \left( \| f \|_2^2 - \| P_t f \|_2^2 \right). \label{eq.12}
\end{equation}

As standard examples, $\lambda = 1$ for the standard Gaussian measure on $\R^n$ and similarly for the discrete cube equipped with any biased measure $\nu$. As a result, from Lemma \ref{lou} and Lemma \ref{lbb}, the spectral gap inequality in its formulation \eqref{eq.89} implies 
$$
\mathcal{S}_{\eta}^{c} (f) \leq (1-\eta) \|f\|_2^2 \quad \mathrm{(resp.} \, \mathcal{S}_{\eta}^{\mathcal{G}} (f) \leq \sqrt{1-\eta^2} \|f\|_2^2)
$$ for (centered) functions of the discrete cube (resp. of the Gaussian space). Thus, for a sequence of (centered) functions $(f_n)_{n \geq 0}$ defined on the discrete cube or on the Gaussian space such that $\inf_{n \geq 0} \|f_n \|_2 >0$, the spectral gap inequality does not predict if it is a noise sensitive sequence.

We then extend the definition of noise stability of sets/functions in our both discrete and continuous setting in order to preserve this property.

Turning to log-concave measures, we give the following definition. 
\begin{definition} \label{nsc}
Let $(\R^n, \mu)$, where $\mu$ is a product a log-concave probability measures, with $\lambda$ its spectral gap constant. For a measurable function $f$ in $L^2(\mu)$, we define its noise stability with parameter $\eta \in [0,1)$ by  
\[ 
\mathcal{S}_{\eta} (f) = \mathrm{Var}_{\mu} (P_t f), 
\]
where $e^{-2 \lambda t} = \sqrt{1-\eta^2}$, 
and similarly for Borel sets $A \subset \R^n$, $\mathcal{S}_{\eta} (A) = \mathcal{S}_{\eta} ({\bf 1}_A)$.
\end{definition}

In the context of Schreier graphs, extending the case of the Boolean cube, the definition of noise stability is the following. 

\begin{definition} \label{nsd}
Let $\Omega$ be a finite Schreier graph with spectral gap constant $\lambda$. For a function $f \, : \, \Omega \to \R$, we define its noise stability with parameter $\eta \in [0,1)$ by
\[
\mathcal{S}_{\eta} (f) = \mathrm{Var}_{\mu} (P_t f), 
\]
where $e^{-2 \lambda t} = 1-\eta$, and similarly for sets $A \subset \Omega$, $\mathcal{S}_{\eta} (A) = \mathcal{S}_{\eta} ({\bf 1}_A)$.
\end{definition}

Notice that, in more probabilistic terms, since $(P_t)_{t \geq 0} = (e^{-t} e^{tK})_{t \geq 0}$ and by symmetry of $(P_t)_{t \geq 0}$ with respect to $\mu$, 
$$
\mathcal{S}_{\eta} (f)  = \E_{\mu \otimes \mu} (f(x)f(y)) - (\E_{\mu} f)^2,
$$
where $y$ is obtained from $x$ by acting randomly $m$ elements of the generating set $S$ where $m \sim \mathcal{P}(2t)$ the Poisson law of parameter $t =t(\eta)$.

From subsection $2.2$, the last definition agrees with $\mathcal{S}_{\eta}^c$ in the case of the Boolean cube. Besides, in the general case the spectral gap inequality leads by \eqref{eq.89} to the same upper bound $\mathcal{S}_{\eta} (f) \leq (1-\eta) \mathrm{Var}_{\mu} (f)$. Thus, in this general case, a quantitative relationship of the form \eqref{eq.noise1} is a similar improvement upon the spectral gap inequality.

\subsection{The hypercontractive tool and the decomposition along ``directions''} 

In this section, we describe the main functional tool which will be used in the
proof of the main results of this work. That is, the hypercontractive property of the underlying semigroups, or in 
its equivalent formulation logarithmic Sobolev inequalities. Besides we present in a fairly abstract setting the common decomposition of the Dirichlet form along ``directions'' with the associated influences with respect to these directions.

Say that $(L, \mu)$ satisfies a logarithmic Sobolev inequality whenever there exists $\rho > 0$ such that
\begin{equation}
\rho \, \mathrm{Ent}_{\mu} (f^2) \leq 2 \, \mathcal{E}(f,f) \label{eq.4}
\end{equation}
for every function $f$ on the Dirichlet domain.
The logarithmic Sobolev constant is the largest $\rho > 0$ such that $\eqref{eq.4}$ holds. 
Since the work of Gross \cite{G} in the continuous setting (cf. \cite{Ba}, \cite{Aal}, \cite{B-G-L}...) and Diaconis and Saloff-Coste \cite{D-SC} in the discrete setting,  
it is known that a logarithmic Sobolev inequality is equivalent 
to hypercontractivity of the semigroup  $(P_t)_{t \geq 0}$ (see [N]) in the sense
that for all $f \in L^p(\mu)$ and all $t >0$, $ 1 < p < q < \infty$ with $p \geq 1 + (q-1)e^{-2\rho t}$,
\begin{equation}
\| P_t f \|_q \leq \|f\|_p. \label{eq.5}
\end{equation}
The Sobolev logarithmic constant can also be referred as the hypercontractive constant. It is classical that it is lower that the spectral gap constant, i.e. $\lambda \geq \rho$. It is also a main feature of these inequalities that
they are stable by product, so that the spectral gap and logarithmic Sobolev constants of
a product space are the minimum of the spectral gap and logarithmic Sobolev constants of each factor in the product.
Recall that $\lambda = \rho = 1$ for the standard Gaussian measure on $\R^n$
and similarly for the uniform measure on the discrete cube.  For the further purposes,
let us emphasize also that if $d\mu(x) = e^{-V(x)} dx$ on $\R^n$
is such that the
Hessian of the potential $V$ satisfies $\mathrm{Hess} (V) \geq c >0$ (uniformly, as symmetric matrices),
then $\lambda \geq  \rho \geq c$. This very classical fact follows from the pioneer work
of Bakry and \'Emery \cite{B-E} around hypercontractive diffusions.

\vskip 2mm

In each class of examples of the above subsections, a key property is the decomposition
of the Dirichlet form along ``directions''. We will be interested into situations in which 
these directions commute in an appropriate sense with derivation. This may be expressed
in the following abstract formulation, that immediately applies to the various examples
of interest. Namely, assume that there is
a decomposition of the Dirichlet form $\mathcal{E}$  along ``directions'' $i$ 
for some operators $\Gamma_i$ such that for any suitable $f$ on $\Omega$,
\begin{equation}
\mathcal{E} (f,f) = \sum_{i=1}^m  \int_{\Omega} \Gamma_i (f)^2 \, d\mu, \label{eq.8}
\end{equation}
where the operators $\Gamma_i$ are commuting with the semigroup in a sense that there 
exists a real constant $\kappa$ such that for every $i  = 1, \ldots, m $
\begin{equation}
\Gamma_i (P_t f) \leq e^{\kappa t} P_t (\Gamma_i f). \label{eq.7}
\end{equation}
To illustrate this rather formal property, consider the example of the Ornstein-Uhlenbeck semigroup
on $\R^n$ for which $\partial_i P_t f= e^{-t} P_t (\partial _i f)$ yielding therefore \eqref {eq.8}
with $\kappa = 1$. In the framework of a log-concave measure $d\mu(x) = e^{-V(x)} dx$, 
it is known (see e.g \cite{Ba}, \cite{B-G-L}) that
whenever the Hessian of the potential $V$ satisfies $\mathrm{Hess} (V) \geq c$
where $ c \in \R$,
for every smooth $f : \R^n \to \R$ and every $t \geq 0$,
\begin{equation} 
|\nabla P_t f| \leq e^{-ct} P_t ( |\nabla f|). \label{eq.6}
\end{equation}
Therefore, a product of log-concave measures $d\mu_i(x) = e^{-V_i(x)} dx$,
$i  = 1, \ldots, m$, for which each potential $V_i$ satisfies
$\mathrm{Hess} (V_i) \geq c_i$ with $ c_i \in \R$, is another instance of the
decomposition \eqref {eq.8} with $\Gamma_i = \nabla_i$ and $\kappa = \min c_i$.
In the discrete product setting, we may take $\Gamma_i = L_i$.
On the cube, the relations  $D_i T_t^p f = T_t^p D_i f$ then ensure the suitable
commutation with $\kappa = 0$.
 
Such general decompositions have been emphasized in \cite{CE-L} in connections with
the study of influences. They are similarly useful here, and together with hypercontractivity
will be at the root of the main conclusions. Indeed, assuming ergodicity and proceeding as in the proof of Lemma \ref{dovg}, 
the noise stability can then be expressed as 
\[
\mathcal{S}_{\eta} (f) =  \sum_{i=1}^m \int_{t}^{\infty} \int_{\Omega} \Gamma_i (P_s f)^2 \, d\mu \, ds, 
\]
where $t = t(\eta)$ is given as in Definition \ref{nsc} in the continuous case or Definition \ref{nsd} in the discrete one. Another example of interest given by the Euclidean spheres will be presented in Section 5 below.

To close this paragraph, it is worthwhile mentioning
that the content of our statements are invariant under translation of the functions
by a constant. Therefore, in the remaining of the paper, it will be implicitly assumed that
all functions are centered.

\section{The Gaussian and log concave settings} 

\vskip 2mm

The section will be devoted to the proof of Theorem~\ref {blg1}, actually in
a more general formulation. 

Consider therefore a probability measure $\mu$ on
$\R^n = \R^{n_1} \times \cdots \times \R^{n_m}$ of the form 
$d\mu(x) = \bigotimes_{i=1}^m e^{-V_i(x)}dx$ with $\mathrm{Hess} (V_i) \geq c > 0$
for every $i = 1, \ldots, m$.
Recall that, if $\rho$ designs the hypercontractive constant and $\lambda$ the Poincar\'e constant, in this setting it holds $\lambda \geq \rho \geq c$. 

\begin{theorem} \label {bl}
Let $\mu$ as above and let
$f : \R^n \to \R$ be $C^1$-smooth and in $L^2(\mu)$. Then, for every $\eta \in (0,1)$, 
$$
\mathcal{S}_{\eta}(f) \leq \mathrm{max}\left(4,\frac{4}{c}\right) (1-\eta^2)^{\frac{c}{4 \lambda}}
\bigg( \sum_{i=1}^m \| \nabla_i f \|_1^2 \bigg)^{\alpha(\eta)} \|f\|_2^{2-2\alpha(\eta)}
$$
where
$$
\alpha(\eta) = \frac{1-(1-\eta^2)^{\frac{\rho}{4 \lambda}}}{2}.
$$
\end{theorem}

To compare with the results of \cite{K-M-S2} in the Gaussian case (corresponding therefore
to $m=n$, $n_1 = \cdots = n_m =1$, and $V_i$ quadratic) for which $\lambda = \rho = c = 1$,
Theorem \ref{bl} indicates that for all $f : \R^n \to \R$ and every $\eta \in (0,1)$, 
$$
\mathcal{S}_{\eta}^{\mathcal{G}}(f) \leq 4 (1-\eta^2)^{1/4} 
   \bigg( \sum_{i=1}^n \|\partial_i f\|_{L^1(\mu)}^2\bigg)^{\frac{1-(1-\eta^2)^{1/4}}{2}} 
   \|f\|_2^{1+(1-\eta^2)^{1/4}}.
$$
It may always be assumed 
that $\sum_{i=1}^n \|\partial_i f\|_{L^1(\mu)}^2 \leq \|f \|_2^2$ otherwise the above inequality is implied by the spectral gap inequality (in its formulation $\mathcal{S}_{\eta}^{\mathcal{G}} (f) \leq \sqrt{1-\eta^2} \|f\|_2^2$). 
As 
\[
1-(1-\eta^2)^{1/4} \geq \frac{\eta^2}{4},
\]
it thus yields the inequality of \cite{K-M-S2} with $C_1 =4$ and
$C_2 = \frac{1}{8}$.

\begin{proof} [Proof of Theorem \ref{bl}]
We make the assumption that $f$ is such that
$$
\sum_{i=1}^m \| \nabla_i f \|_1^2  \leq \|f\|_2^2,
$$
otherwise there is nothing to prove (again, by the spectral gap inequality). We set $t = -\frac{\log (1-\eta^2)}{4 \lambda}$, 
so that $\mathcal{S}_{\eta}(f) = \mathrm{Var}_{\mu}(P_t f)$.

We start with Lemma \ref{dovg} from which
$$
\mathrm{Var}_{\mu}(P_t f) = 2 \int_t^{\infty} \sum_{i=1}^m \int_{\R^n} |\nabla_i P_s f|^2 d\mu \, ds. 
$$
The main step of the proof consists in the following cut-off argument.
Namely, for $i = 1, \ldots, m$ and $M>0$,
\begin {equation} \label {eq.cutoff}
\int_{\R^n} |\nabla_i P_s f|^2 d\mu  
=  \int\limits_{\{ | \nabla_i P_s f | \leq M \|\nabla_i f \|_1 \}} |\nabla_i P_s f|^2 d\mu 
+ \int\limits_{\{ | \nabla_i P_s f | > M \|\nabla_i f \|_1 \}} |\nabla_i P_s f|^2 d\mu.
\end {equation} 
The first integral is bounded from above as
\begin{eqnarray*}
\int\limits_{\{ | \nabla_i P_s f | \leq M \|\nabla_i f \|_1 \}} |\nabla_i P_s f|^2 d\mu &\leq & M \|\nabla_i f \|_1 \int_{\R^n} |\nabla_i P_s f| d\mu  \\
&\leq & M \|\nabla_i f \|_1 e^{-cs} \int_{\R^n} P_s (|\nabla_i  f|) d\mu \\
\end{eqnarray*}
where we use in the last inequality the commutation property $\eqref{eq.6}$ 
which ensures that $|\nabla_i P_s f| \leq e^{-cs} P_s( |\nabla_i f|)$. 
Since the measure $\mu$ is invariant with respect to $(P_s)_{s \geq 0}$, it follows that
$$
\int\limits_{\{ | \nabla_i P_s f |
\leq M \|\nabla_i f \|_1 \}} |\nabla_i P_s f|^2  d\mu \leq e^{-cs} M \|\nabla_i f \|_1^2.
$$
After integrating in time and summing over $i = 1, \ldots, m$, we reach a first bound
\begin{equation} \label{eq.21}
\sum_{i=1}^m \int_t^{\infty}  \int\limits_{\{ | \nabla_i P_s f | 
\leq M \|\nabla_i f \|_1 \}} |\nabla_i P_s f|^2 d\mu \, ds
\leq \frac{1}{c} e^{-ct} M \sum_{i=1}^m \| \nabla_i f \|_1^2.
\end{equation}

\vskip 2mm 

We now focus on bounding from above 
$$
\sum_{i=1}^m \int\limits_{\{ |\nabla_i P_s f | > M \| \nabla_i f \|_1 \}} |\nabla_i P_s f|^2 d\mu.
$$
To this task,
for every $ i =  1, \ldots, m $, H{\"o}lder's inequality applied to $ |\nabla_i P_s f|^2 $ and 
${\bf 1}_{ \{ | \nabla_i P_s f | > M \|\nabla_i f \|_1 \} }$ yields 
\begin{equation} \label{eq.1989}
\int\limits_{\{ | \nabla_i P_s f | > M \|\nabla_i f \|_1 \}} |\nabla_i P_s f|^2 d\mu 
\leq \mu \big \{ | \nabla_i P_s f | > M \|\nabla_i f \|_1 \big \}^{\frac{1}{q}} \| \nabla_i P_s f \|_{2p}^2
\end{equation}
for every $p, q \geq 1$ such that $\frac{1}{p} + \frac{1}{q} = 1$. Yet, since $P_s f = P_{s/2} ( P_{s/2} f)$, using $\eqref{eq.6}$ once more,
$$ 
|\nabla_i P_s f |^{2p} \leq  e^{-cps}   [P_{s/2} (|\nabla_i P_{s/2} f|)]^{2p},
$$
so that by integration
$$ 
\| \nabla_i P_s f \|_{2p}^2 \leq  e^{-cs} \big \|  P_{s/2} (|\nabla_i P_{s/2} f|) \big \|_{2p}^2.
$$
Now,
the hypercontractive property $\eqref{eq.5}$ ensures that, if $\rho$ designs the hypercontractive constant, 
for $p = p(s)$  with $ 2p(s) - 1~=~e^{\rho s},$  
$$
\big \|  P_{s/2} (|\nabla_i P_{s/2} f|) \big \|_{2p}^2 \leq \|\nabla_i P_{s/2} f \|_2^2. 
$$ 
Putting \eqref{eq.1989} and the last two inequalities together, we infer that 
\begin{equation} \label{eq.98}
\int\limits_{\{ | \nabla_i P_s f | > M \|\nabla_i f \|_1 \}} |\nabla_i P_s f|^2 d\mu 
\leq e^{-cs}  \mu \big \{ | \nabla_i P_s f | > M \|\nabla_i f \|_1 \big\}^{\frac{1}{q(s)}}
   \|\nabla_i P_{s/2} f \|_2^2  
\end{equation}
with
$$ 
q(s) = \frac{p(s)}{p(s) -1} = \frac{e^{\rho s} +1}{e^{\rho s} -1}.
$$
By Markov's inequality, we further obtain, using again the commutation $\eqref{eq.6}$
and the invariant property of the semigroup, that
\begin{equation} \label{eq.99}
 \mu \big \{ | \nabla_i P_s f | > M \|\nabla_i f \|_1 \big \} \leq 
\frac{1}{ M \|\nabla_i f \|_1} \int_{\R^n} |\nabla_i P_s f| d\mu \leq \frac{e^{-cs}}{M}.
\end{equation} 
The above bounds \eqref{eq.98} and \eqref{eq.99} therefore imply
\begin{eqnarray*}
\int_t^{\infty} \int\limits_{\{ | \nabla_i P_s f | > M \|\nabla_i f \|_1 \}} |\nabla_i P_s f|^2 d\mu \, ds 
& \leq & \int_t^{\infty} \left( \frac{e^{-c s}}{M} \right)^{\frac{1}{q(s)}}  e^{-c s} \|\nabla_i P_{s/2} f \|_2^2 \, ds \\
& \leq &  e^{-c t}  \int_t^{\infty} \frac{1}{M^{\frac{1}{q(s)}}} \|\nabla_i P_{s/2} f \|_2^2 \, ds.
\end{eqnarray*}
We then notice that the function
$$ 
s \mapsto \frac{1}{q(s)} = \mathrm{tanh} (\rho s/2)
$$
is increasing. Hence, for every $M \geq 1$, 
\[
e^{-ct} \int_t^{\infty} \frac{1}{M^{\frac{1}{q(s)}}}  \|\nabla_i P_{s/2} f \|_2^2 \, ds 
 \leq   \frac{e^{-ct}}{M^{\frac{1}{q(t)}}} \int_t^{\infty}  \|\nabla_i P_{s/2} f \|_2^2 \, ds.
\]
Summing over $i = 1, \ldots, m$,  
\[
\sum_{i=1}^m \int_t^{\infty} \int\limits_{\{ | \nabla_i P_s f | > M \|\nabla_i f \|_1 \}} |\nabla_i P_s f|^2 d\mu \, ds 
\leq  \frac{e^{-ct}}{M^{\frac{1}{q(t)}}} \sum_{i=1}^m \int_t^{\infty}  \|\nabla_i P_{s/2} f \|_2^2 ds.
\]
By Lemma \ref{dovg} again,
$$
\sum_{i=1}^m \int_t^{\infty}  \|\nabla_i P_{s/2} f \|_2^2 ds = \mathrm{Var}_{\mu} (P_{t/2} f) \leq  \| f \|_2^2
$$
(recall that $f$ is centered) so that
\begin{equation} \label{eq.22}
\sum_{i=1}^m \int_t^{\infty} \int\limits_{\{ | \nabla_i P_s f | > M \|\nabla_i f \|_1 \}} |\nabla_i P_s f|^2 d\mu \, ds  \leq   \frac{e^{-ct}}{M^{\frac{1}{q(t)}}} \| f \|_2^2.
\end{equation}

By the decomposition \eqref {eq.cutoff}, 
the two bounds \eqref{eq.21} and \eqref{eq.22} therefore yield 
$$
\mathrm{Var}_{\mu}(P_t f) \leq\mathrm{max}\left(2,\frac{2}{c}\right) e^{-ct} \left(M \sum_{i=1}^m \| \nabla_i f \|_1^2  + \frac{1}{M^{\frac{1}{q(t)}}} \| f \|_2^2 \right).
$$ 
Here, we recall that $M \geq 1$. Given the assumption $\| f \|_2^2 \geq  \sum_{i=1}^m \| \nabla_i f \|_1^2$,
we can choose $M$ such that
$$
M \sum_{i=1}^m \| \nabla_i f \|_1^2 = \frac{1}{M^{\frac{1}{q(t)}}} \| f \|_2^2.
$$ 
We therefore get
$$
 M^{1 + \frac{1}{q(t)}} = \frac{\| f \|_2^2}{\sum_{i=1}^m \| \nabla_i f \|_1^2},
 $$
so that, finally
$$
\mathrm{Var}_{\mu}(P_t f) \leq \mathrm{max}\left( 4,\frac{4}{c} \right) e^{-ct} \left( \sum_{i=1}^n \| \nabla_i f \|_1^2 \right)^{\frac{1}{1+q(t)}} \|f\|_2^{2\frac{q(t)}{1+q(t)}}.
$$ 
Replacing $q(t)$ by its explicit form, and recalling $t = -\frac{\log (1-\eta^2)}{4 \lambda}$, we conclude to the announced claim. 
Theorem~\ref {bl} is established.  
\end{proof}

\section{The discrete setting}

\subsection{The case of Boolean cube and other discrete product spaces}

This section develops the corresponding analysis on the cube and discrete product models.

To deal with the Boolean cube $\{-1,1\}^n$, consider the discrete product structure
as emphasized in Subsection 2.3 consisting of 
a product space
$\Omega = \Omega_1 \times \cdots \times = \Omega_n$ with product probability measure 
$\mu = \mu_1 \otimes \cdots \otimes \mu_n$, each factor $(\Omega_i, \mu_i)$ being endowed
with the Markov operator $ L_i f = \int_{\Omega_i} f d\mu_i - f$.
We recall that the Dirichlet form $\mathcal{E}$ admits the decomposition
$$
\mathcal{E} (f,f) = \sum_{i=1}^n  \int_{\Omega_i} L_i(f)^2 d\mu_i,
$$ 
and that the equality $\textrm{Var}_{\mu} f = \mathcal{E}(f,f)$
implies that the spectral gap constant is equal to 1. 
The underlying (product) semigroup $(P_t)_{t \geq 0}$ will be assumed
to be hypercontractive with constant $\rho$ (equivalently, each
$(L_i, \mu_i)$ is hypercontractive with constant $\rho$). 

\vskip 2mm

\begin{theorem}  \label {bl2}
In the preceding setting,
let $f : \Omega \to \R,$ such that $f \in L^2(\mu)$. Then, if $\eta \in [0, 1)$,
$$
\mathcal{S}_{\eta} (f)  \leq 7 \left( \sum_{i=1}^n \| L_i f \|_1^2 \right)^{\alpha(\eta)}
 \|f\|_2^{2-2\alpha(\eta)}
$$
where
$$
\alpha(\eta) = \frac{1-(1-\eta)^{\rho/2}}{2} \geq \frac{\rho}{4} \eta.
$$
\end{theorem}

Theorem \ref{bl2} contains the result of \cite{K-K} for the discrete cube
$\Omega = \{-1,1\}^n $ with $\mu = \nu_p = (p \delta_{-1} + q\delta_1)^{\otimes n}$.
To make the connection, notice that we can assume that
$$  \sum_{i=1}^n \| L_i f \|_1^2  \leq \|f\|_2^2,$$
otherwise there is nothing to prove. With the previous notations,
it is easily seen that for every $i \in \{1, \ldots, n\}$,
$$ 
\| L_i f \|_1 = 2p q \| D_i f \|_1 = 2pq I_i(f).
$$ 
Thereby, we get 
$$
\mathcal{S}_{\eta}^c (f) \leq 7(2p q)^{\frac{\rho}{2} \eta} \left( \sum_{i=1}^n (I_i f)^2 \right)^{\frac{\rho}{4} \eta } \|f\|_2^{2-\frac{\rho}{2} \eta}. 
$$ 
so that, since $4pq \leq 1$,
$$
\mathcal{S}_{\eta}^c (f) \leq 7 \left( \sum_{i=1}^n (I_i f)^2 \right)^{\frac{\rho}{4} \eta} \|f\|_2^{2-\frac{\rho}{2} \eta }. 
$$
On the other hand, it is known (see \cite{CE-L}) the logarithmic Sobolev constant is equal to
$$\rho = \frac{2(p-q)}{\log p - \log q} \quad (=1 \, \mathrm{if} \, p=q).$$
which amounts to the main result of \cite{K-K} as emphasized in the introduction.
In the uniform case (i.e $p=q$), we have $\rho =1$ and the above result is similar to Theorem 4 
of \cite{K-K} with weaker assumptions on $f$. In the biased case, the constants are somewhat weaker for small $p$ or $q$.  

It is worth to point out that the quantitative relationship of Theorem~\ref {bl2}
yields empty results when $\log p$ is of order $\log n$. 
In this range indeed, the theorem does not hold even qualitatively (see \cite{K-K}).

\begin{proof} [Proof of Theorem  \ref{bl2}]

The scheme of proof is completely
similar to the one developed in the continuous setting for Theorem~\ref {bl}.
We nevertheless introduce a further trick, replacing
the decomposition of the variance along the semigroup of Lemma \ref{dovc} by its
restriction to a finite domain of integration in time $[0,T]$ for some $T \geq 0$ in the form
of the identity
$$ 
\| f \|_2^2 - \| P_T f \|_2^2 = - \int_0^T \frac{d}{ds} \|P_s f \|_2^2 \, ds 
= 2 \int_0^{T} \mathcal{E} ( P_s f, P_s f ) ds
= 2 \int_0^{T} \sum_{i=1}^n \int_{\Omega} |L_i P_s f|^2 d\mu \, ds.
$$
Recall that if $e^{-2t} = 1 - \eta,$ $\mathcal{S}_{\eta}(f) = \mathrm{Var}_{\mu}(P_t f)$. Hence, by \eqref{eq.12}, for every $T >0$,
$$
\mathcal{S}_{\eta}(f) \leq \frac{2}{1-e^{-T}} 
\int_t^{t+T} \sum_{i=1}^n \int_{\Omega} |L_i P_s f|^2 d\mu \, ds. 
$$

For each $i = 1, \ldots, n$, we again cut
the integral into two parts with $M \geq 1$. The same commutation and contraction argument yield
$$
\int\limits_{\{ | L_i P_s f | \leq M \|L_i f \|_1 \}} |L_i P_s f|^2 d\mu  
\leq  M \|L_i f \|_1 \int_{\Omega} |L_i P_s f| d\mu  \leq M \|L_i f \|_1^2. 
$$
Therefore
$$
\mathrm{Var}_{\mu}(P_t f) \leq  \frac{2T}{1-e^{-T}} \bigg( M \sum_{i=1}^n \|L_i f \|_1^2  +  \frac{1}{T}\int_t^{t+T} \sum_{i=1}^n 
\int\limits_{\{ | L_i P_s f | > M \|L_i f \|_1 \}} |L_i P_s f|^2 d\mu \, ds \bigg).
$$

Above the truncation level, for each $i = 1, \ldots , n$, the same argument
 based on the H\"older and hypercontractivity inequalities yields
\begin{eqnarray*}
\int\limits_{\{ | L_i P_s f | > M \|L_i f \|_1 \}} |L_i P_s f|^2 d\mu 
  & \leq & \mu \big \{ | L_i P_s f | > M \|L_i f \|_1  \big \}^{\frac{1}{q}} \| L_i P_s f \|_{2p}^2 \\
 & \leq &  \|L_i P_{s/2} f \|_2^2  \, \mu \big \{ | L_i P_s f | > M \|L_i f \|_1 \big \}^{\frac{1}{q(s)}} \\
 & \leq & \frac{1}{M^{\frac{1}{q(s)}}}  \| L_i P_{s/2} f \|_2^2,
\end{eqnarray*}
with $ q(s) = \frac{e^{\rho s} +1}{e^{\rho s} -1}.$ Therefore,
$$ 
\int_t^{t+T} \sum_{i=1}^n \int\limits_{\{ | L_i P_s f | > M \|L_i f \|_1 \}} |L_i P_s f|^2 d\mu \, ds \leq \frac{1}{M^{\frac{1}{q(t)}}} 
\sum_{i=1}^n   \int_t^{t+T} \| L_i P_{s/2} f \|_2^2 \, ds.
$$ 
Since
$$
\sum_{i=1}^n   \int_t^{t+T} \| L_i P_{s/2} f \|_2^2 \, ds
 \leq \sum_{i=1}^n   \int_0^{\infty} \| L_i P_{s/2} f \|_2^2 \, ds = \|f\|_2^2,
$$
we thus get that
$$
\mathrm{Var}_{\mu}(P_t f) \leq\frac{2 \mathrm{max} (1,T)}{1-e^{-T}} \left( M \bigg( \sum_{i=1}^n \|L_i f\|_1^2 \bigg)  + \|f\|_2^2  \frac{1}{M^{\frac{1}{q(t)}}} \right). 
$$ 
The theorem then follows as in the conclusion of the proof of Theorem~\ref {bl}, using
moreover the fact that
$$
\inf_{T >0} \frac{4 \mathrm{max} (1,T)}{1-e^{-T}} = \frac{4}{1-e^{-1}} \leq 7.
$$ 
\end{proof}

\subsection{The case of more general Schreier graphs and non products examples}

This subsection briefly discusses further examples of interest, basically non-product models, for which
the preceding approach may be developed similarly. The basic ingredients for such extension
are the decomposition of the variance into directional derivatives and hypercontractivity.

\vspace{2mm}
Among discrete examples, the recent work of \cite{O-W} investigates the examples of general Schreier or Cayley graphs. 
In the context of Subsection 2.3, recall that the Dirichlet form $\mathcal{E}$ takes the form
$$
\mathcal{E}(f,f) = \frac{1}{2|S|} \sum_{x \in \Omega} \sum_{s \in S} [f(x^s) - f(x)]^2 \mu(x) = \frac{1}{2|S|} \sum_{s \in S} \|D_s f\|_{L^2(\Omega)}^2.
$$
It is shown in \cite{O-W} that the commutation $\eqref{eq.7}$ holds with $\kappa = 0$ (see also \cite{CE-L}).
Denote as usual by $(P_t)_{t \geq 0}$ the underlying semigroup attached to
this Dirichlet form, and let $\lambda$ be the spectral gap constant and $\rho$
the logarithmic Sobolev constant. Noticing that there is a positive constant $C$ such that $\textrm{inf}_{T >0} \frac{\textrm{max}(1,T)}{1-e^{- \lambda T}} \leq \frac{C}{\lambda}$, we get the following theorem
from the general scheme of proof developed in the preceding subsection.

\begin{theorem} \label {corg}
Let $\Omega$ be a Schreier or Cayley graph and $f \, : \, \Omega \to \R$. Then, for any $t \geq 0$,
$$
\mathcal{S}_{\eta}(f) \leq \frac{C}{\lambda 
} \left(\frac{1}{2|S|} \sum_{s \in \mathcal{S}}  I_s(f)^2 \right)^{\alpha(\eta)} \|f\|_2^{2-2\alpha(\eta)},
$$
for some numerical constant $C>0$ and
$$
\alpha(\eta) = \frac{1 - (1-\eta)^{\rho/(2\lambda)}}{2} \geq \frac{\rho}{4 \lambda} \eta.
$$
\end{theorem}

Thus, for a sequence of graphs $(\Omega_n)_{n \geq 0}$, we see from Theorem \ref{corg} that the original Benjamini--Kalai--Schramm criterion holds whenever $\inf_{n \in \N} \frac{\rho_n}{\lambda_n} >0$, since in this case there exists an universal constant $c >0$ such that $\alpha_n(\eta) \geq c \eta$. Besides, Theorem \ref{corg} represents then a quantitative form similar to \eqref{eq.noise1}.

Explicit examples with known respective spectral gap and logarithmic Sobolev constants are given in \cite{O-W}. Among them, we can cite the discrete tori $(\Z/m\Z)^n$, $m \geq 2$. Then, given the product structure, both constants $\lambda_n$ and $\rho_n$ are of the same order, so that Theorem \ref{corg} provides a quantitative B-K-S relationship. The case of the Boolean cube $C_n$ with uniform measure can be seen as the Cayley graph $(\Z/2\Z)^n$ generated by its canonical basis $S = (e_i)_{1 \leq n}$. Then the Dirichlet form is 
$$
\mathcal{E}(f,f) = \frac{1}{2n} \sum_{i =1}^n \|D_i f\|_{L^2(C_n)}^2,
$$
so that the spectral gap constant and the hypercontractive constant are both equal to $\frac{2}{n}$ with this normalization. Rescaling by multiplying the Dirichlet form by $\frac{n}{2}$, we see that the above statement is in this case the one of \cite{K-K} and is therefore contained in Theorem \ref{bl2}. 

The main novelty with respect to the previous subsection is that it also covers two non product models of graphs,
namely the symmetric group and the slices of the Boolean cube. In the case of the symmetric group $\mathfrak{S}_n$, recall its
Cayley graph structure with the generating set given by the subgroup
of the transpositions $\mathcal{T}_n$. The spectral gap $\lambda_n$ is equal to $\frac{2}{n-1}$,
and the logarithmic Sobolev constant $\rho_n$ is greater than $\frac{a}{n \log n}$ for some $a >0$ (see \cite{D-SC}).
Hence, Theorem~\ref{corg} does not improve upon the spectral gap inequality since $\frac{\rho_n}{\lambda_n}$ goes to $0$ as $n$ goes to infinity. Notice that this conclusion also holds for the inequalities established in \cite{CE-L} and \cite{O-W} in the case of the symmetric group.

However, as pointed out in \cite{O-W}, in the case of the the slices of the Boolean cube, both spectral gap and Sobolev logarithmic constants are of the same order.
For $n \geq 1$, $1 \leq k < n$, the slices
(of order $k$) of the Boolean cube are defined
by ${[n] \choose k}  = \{x \in \{0,1\}^n, \sum_{i=1}^n x_i =k \}$. The symmetric group is acting on ${[n] \choose k}$ by $x^{\sigma} = (x_{\sigma(i)})_{1 \leq i \leq n}$, 
so that it has a Schreier graph structure. The generators are given by the transpositions
$\tau_{ij}$, $1 \leq i < j \leq n$, with $x^{\tau_{ij}}$ being obtained from $x$ by switching $i$ and $j$. Then, it is a result of Lee and Yau (\cite{L-Y}) that the spectral gap constant is $\lambda = \frac{1}{n}$ and the logarithmic Sobolev constant $\rho$ satisfies $\rho^{-1} \sim n \log \frac{n^2}{k(n-k)}$. In particular when
$ \frac {k}{n}$ is bounded away from $0$ and $1$, both constants are of the same order. We again rescale the Dirichlet form by multiplying by $n$ so that 
$$
\mathcal{E}'(f,f)  = \frac{1}{n-1} \sum_{1 \leq i < j \leq n } \|D_{\tau_{ij}} f\|_2^2.
$$
Thus, the spectral gap constant associated to $\mathcal{E}'$ is equal to $1$ and the according Sobolev logarithmic constant $\rho'$ satisfies  $\rho'^{-1} \sim  \log \frac{n^2}{k(n-k)}$. In particular, whenever $\frac {k}{n}$ is bounded away from $0$ and $1$, $\rho'$ is bounded away from $0$. Denoting by $\mathcal{S}_{\eta}^s = \mathcal{S}_{\eta}^{s_k}$ the noise stability in this context, Theorem \ref{corg} implies therefore that
\begin{equation} \label{bksslices}
\mathcal{S}_{\eta}^s f  \leq 7  
\bigg(\frac{1}{n-1} \sum_{1 \leq i < j \leq n}  (I_{\tau_{ij}} f)^2 \bigg)^{c \eta} \|f\|_2^{2 - 2c\eta},
\end{equation} 
for some positive constant $c$. We recall that $I_{\tau_ {ij}} f$ is the influence 
of the transposition $\tau_{ij}$ on $f$. 
Therefore, the original \cite{B-K-S} criterion holds over slices of the Boolean cube of order $k \in [c'n, (1-c')n]$, for each positive constant $c'$. That is, if the sequence $\bigg( \frac{1}{n_{\ell}-1} \sum_{1 \leq i < j \leq n_{\ell}}  I_{\tau_{ij}}^2 f_{n_{\ell}} \bigg)_{n_{\ell} \geq 0}$ goes to $0$ when $n_{\ell}$ goes to infinity, then the sequence of $(f_{n_{\ell}})_{n_\ell \geq 0}$ is (asymptotically) noise sensitive. 
It is worth mentioning that a variant of this qualitative result over the slices has been recently established in \cite{Fo}, where the author uses this result in connection with the notion of exclusion sensitivity. We point out that \eqref{bksslices} represents a quantitative version similar to the results of \cite{K-K}, \cite{K-M-S2}. One would be interesting to provide any application of this quantitative result.

\section{The case of the Euclidean spheres}

To conclude, we present the continuous model given by the Euclidean spheres in which the preceding scheme of proof apply. That is, it holds a decomposition $\eqref{eq.8}$ with commutation $\eqref{eq.7}$ as well as hypercontractivity $\eqref{eq.5}$ and spectral gap $\eqref{eq.3}$ (see \cite{CE-L}).  Although we could give a similar definition of noise stability, we will not do so as it has not a clear signification. We will express the results only in term of the heat semigroup associated to the spherical Laplacian.

Let $\mathbb{S}^{n-1} \subset \R^n$ ($n \geq 2$), the $(n-1)$-dimensionnal Euclidean sphere equipped with its normalized surface measure $\mu$. Consider for $i,j \in \{1, \ldots, n\}$,
$ \, D_{i,j} = x_j \partial_i - x_i \partial_j$. 
The Dirichlet form associated with the spherical Laplacian 
$$
\Delta = \frac{1}{2} \sum_{i,j = 1}^n D_{i,j}^2
$$ 
takes the form
$$
 \mathcal{E}(f,f) = \int_{\mathbb{S}^{n-1}} f(- \mathrm{\Delta} f)d\mu = \frac{1}{2} \sum_{i,j=1}^n \int_{\mathbb{S}^{n-1}} (D_{i,j}f)^2 d\mu.
 $$  
Since we clearly have $\Delta D_{i,j} = D_{i,j} \Delta$, $\eqref{eq.7}$ holds with $\kappa = 0$. 
Consider again the heat semigroup $(P_t f)_{t \geq 0} = (e^{t \Delta} f)_{t \geq 0}$
generated by the Laplacian $\Delta$.
It is known (see e.g. \cite{Ba}) that the spectral gap constant and the logarithmic Sobolev constant are both equal to $n-1$. Our result can then be stated as follows.  

\begin{theorem} \label{tsph}
Let $f \, : \,\mathbb{S}^{n-1} \to \R,$ where $n \geq 2$ be $C^1$-smooth and in $L^2(\mu)$.
Then, for any $t \geq 0$,  
$$
\mathrm{Var}_{\mu}(P_t f) \leq 7 \bigg( \sum_{i,j=1}^n \| D_{i,j} f \|_1^2 \bigg)^{\alpha(t)} \|f\|_2^{2-2\alpha(t)},
$$
where 
$$
\alpha(t) = \frac{1-e^{-(n-1)t}}{2}.
$$
\end{theorem}

\begin{proof}
The proof follows the scheme  of the one of Theorem~\ref {bl},
with however the twist emphasized in the discrete framework of Section~4, i.e.
restricting to a finite domain of integration in time $[0,T]$ for some $T \geq 0$
using \eqref {eq.12} and replacing the logarithmic Sobolev constant by its value $n-1$. 
Then, we just notice that for every $\lambda \geq 1$ 
$$
\inf_{T >0} \frac{4 \, \mathrm{max} (1,T)}{1-e^{-\lambda T}} = \frac{4}{1-e^{-\lambda}} \leq 7.
$$ 
\end{proof} 

Recall that the Poincar\'e inequality implies that
$\textrm{Var}_{\mu}(P_t f) \leq e^{-(n-1)t} \textrm{Var}_{\mu}(f)$.
Since the logarithmic Sobolev constant and the spectral gap constant are equal
in this case, Theorem~\ref {tsph} provides some non-trivial bound on $\textrm{Var}_{\mu}(P_t f)$
in specific ranges of $t$ (i.e. when $t$ of order $\frac {1}{n}$). To the best of our knowledge, this inequality is new. 
It would therefore be of interest to exhibit functions $(f_n)_{n \geq 0}$ over the Euclidean spheres such that the sequence $\bigg( \sum_{i,j=1}^n \| D_{i,j}f_n\|_1^2  \bigg)_{n \geq 2}$ goes to $0$.

\vskip 5 mm 

\textit{Acknowledgement. This work has been completed when I made my Ph.D in the University of Toulouse. I warmly thank my Ph.D advisor Michel Ledoux for introducing this problem to me, and for fruitful discussions. I also thank the anonymous referee for helpful comments in improving the exposition.
}

\vskip 10 mm

\noindent
\textsc{Rapha{\"e}l Bouyrie,} \\
\textsc{\small{Laboratoire d'Analyse de Math\'ematiques Appliqu\'es, UMR 8050 du CNRS, Universit\'e Paris-Est Marne-la-Vall\'ee, 5 Bd Descartes, Champs-sur-Marne, 77454 Marne-la-Vall\'ee Cedex, France}} \\
\textit{E-mail address:} \texttt{raphael.bouyrie@upem.fr}
 
\end{document}